\documentclass[12pt]{article}
\usepackage{amsmath, amssymb, amsfonts}
\usepackage{amsthm}
\usepackage[utf8]{inputenc}
\usepackage{csquotes}
\usepackage{geometry}
\usepackage{adjustbox}
\usepackage{graphicx}
\usepackage{verbatim}
\usepackage{array}
\usepackage{enumitem}
\usepackage{booktabs}
\usepackage{hyperref}
\usepackage{url}
\usepackage[capitalize]{cleveref}
\usepackage{tikz}
\usetikzlibrary{positioning, shapes.geometric, arrows.meta}

\title{One-way multilinear functions of the second order with linear shifts}
\author{Stanislav Semenov \\
\href{mailto:stas.semenov@gmail.com}{stas.semenov@gmail.com} \\
\href{https://orcid.org/0000-0002-5891-8119}{ORCID: 0000-0002-5891-8119}}
\date{June 19, 2025}

\theoremstyle{definition}
\newtheorem{definition}{Definition}[section]

\newtheorem{conjecture}[definition]{Conjecture}

\theoremstyle{plain}

\newtheorem{proposition}[definition]{Proposition}

\theoremstyle{remark}

\begin{document}

\maketitle

\begin{abstract}
We introduce and analyze a novel class of binary operations on finite-dimensional vector spaces over a field \( K \), defined by second-order multilinear expressions with linear shifts. These operations generate polynomials whose degree increases linearly with each iterated application, while the number of distinct monomials grows combinatorially. We demonstrate that, despite being non-associative and non-commutative in general, these operations exhibit \emph{power associativity} and \emph{internal commutativity} when iterated on a single vector. This ensures that exponentiation \( a^n \) is well-defined and unambiguous.

Crucially, the absence of a closed-form expression for \( a^n \) suggests a \emph{one-way property}: computing \( a^n \) from \( a \) and \( n \) is efficient, while recovering \( n \) from \( a^n \) (the \emph{Discrete Iteration Problem}) appears computationally hard. We propose a \emph{Diffie–Hellman-like key exchange protocol} based on this principle, introducing the \emph{Algebraic Diffie–Hellman Problem (ADHP)} as an underlying assumption of security.

In addition to the algebraic foundations, we empirically investigate the orbit structure of these operations over finite fields, observing frequent emergence of long cycles and highly regular behavior across parameter sets. Motivated by these dynamics, we further propose a pseudorandom number generation (PRNG) strategy based on multi-element multiplication patterns. This approach empirically achieves near-maximal cycle lengths and excellent statistical uniformity, highlighting the potential of these operations for cryptographic and combinatorial applications.
\end{abstract}

\subsection*{Mathematics Subject Classification}
17A30 (Algebras satisfying identities), 15A75 (Exterior algebra; multilinear algebra)

\subsection*{ACM Classification}
I.1 Symbolic and Algebraic Manipulation, F.2 Analysis of Algorithms and Problem Complexity, E.3 Data Encryption

\section*{Introduction}

We introduce a class of binary operations \( * : V \times V \to V \) on finite-dimensional vector spaces \( V = K^n \) over a field \( K \). These operations are defined componentwise through second-order multilinear expressions, augmented with linear shifts. A prototype of such an operation was first proposed in~\cite{semenov2025stratified}. A distinctive aspect of these operations is how algebraic complexity scales with iterated applications. Specifically, while the total polynomial degree of \( a^n := a * a * \dots * a \) increases linearly with \( n \), the number of distinct monomials involved grows combinatorially, making the derivation of a general closed-form expression highly non-trivial.

This inherent complexity leads to a central theme of this work: the system's one-way characteristics~\cite{Goldreich2001}. We observe that computing \( a^n \) for given \( a \) and \( n \) is computationally efficient. However, the reverse problem—determining \( n \) from \( a^n \) (which we formally define as the Discrete Iteration Problem)—appears to be computationally intractable, particularly for general parameters and large \( n \). This algebraic asymmetry is a cornerstone of our investigation.

The study of such structures integrates concepts from non-associative algebra, symbolic computation, and computational complexity. In this paper, we first precisely define these operations and provide explicit componentwise formulas to demonstrate their second-order nature and the rapid escalation of complexity with iteration. We then rigorously examine the recursive behavior and fundamental algebraic properties, including the crucial observations of power associativity and internal commutativity that emerge for powers of a single element. Finally, we leverage these properties and the conjectured hardness of the Discrete Iteration Problem to propose a Diffie–Hellman-like key exchange protocol~\cite{Diffie1976} operating over finite fields, thereby introducing the Algebraic Diffie–Hellman Problem (ADHP). This framework holds significant potential for the development of new cryptographic primitives, as well as advancing research in algebraic dynamics and computational algebra.

\section{Multilinear Operation on \texorpdfstring{$K^3$}{K³}}

Consider the operation \( * : V \times V \to V \), where \( V = K^3 \), defined component-wise by the following rule:
\begin{align*}
(ab)_0 &= a_0 + b_0 + a_0 b_0 + A a_1 b_1 + C a_2 b_1 + B a_2 b_2, \\
(ab)_1 &= a_1 + b_1 + a_1 b_0 + a_0 b_1 + D a_1 b_1 + E a_1 b_2, \\
(ab)_2 &= a_2 + b_2 + a_2 b_0 + a_0 b_2 + D a_2 b_1 + E a_2 b_2,
\end{align*}
where \( A, B, C, D, E \in K \) are fixed parameters.

The explicit form of the product vector \( a * a := a^2 \) is given by:
\[
a^2 =
\begin{pmatrix}
(a_0 + 1)^2 + A a^2_1 + B a^2_2 + C a_1 a_2 -1 \\
a_1 (D a_1 + E a_2 + 2 (a_0 + 1)) \\
a_2 (D a_1 + E a_2 + 2 (a_0 + 1))
\end{pmatrix}.
\]

\subsection*{Symbolic Expansion of \texorpdfstring{$a^3$}{a³}}

Using symbolic computation, we obtain the following expansion for the zeroth component of \( a^3 := a * a * a \):
\begin{align*}
(a^3)_0 &= (a_0 + 1)^3 + A D a_1^3 + B E a_2^3 \\
&\quad + 3 a_0 (A a_1^2 + B a_2^2) + 3 C a_0 a_1 a_2 \\
&\quad + (A E + C D) a_1^2 a_2 + (B D + C E) a_1 a_2^2 \\
&\quad + 3 A a_1^2 + 3 B a_2^2 + 3 C a_1 a_2 - 1, \\
(a^3)_1 &= a_1 (3 (a_0 + 1)^2 + (A + D^2) a_1^2 + (B + E^2) a_2^2 \\
&\quad + 3 (D a_1 + E a_2) (a_0 + 1) + (C + 2 D E) a_1 a_2, \\
(a^3)_2 &= a_2 (3 (a_0 + 1)^2 + (A + D^2) a_1^2 + (B + E^2) a_2^2 \\
&\quad + 3 (D a_1 + E a_2) (a_0 + 1) + (C + 2 D E) a_1 a_2.
\end{align*}

These expressions illustrate how cubic terms naturally emerge from the iterated application of the operation. With each iteration, the total polynomial degree increases linearly, while the number and diversity of monomials grow combinatorially. Each coefficient depends on the parameters \( A, B, C, D, E \), reflecting complex interactions among the components \( a_0, a_1, a_2 \).

\subsection*{Functional Structure and Recursive Dependence}

The explicit form of the product vector \( a^2 \) admits a simplified functional representation. Define two scalar functions:
\[
g(a) := (a_0 + 1)^2 + A a_1^2 + B a_2^2 + C a_1 a_2 - 1, \qquad
h(a) := D a_1 + E a_2 + 2(a_0 + 1).
\]
Then the result of the self-product can be expressed compactly as:
\[
a^2 =
\begin{pmatrix}
g(a) \\
a_1 \cdot h(a) \\
a_2 \cdot h(a)
\end{pmatrix}.
\]

This formulation reveals a nested compositional structure: computing \( a^3 = a^2 * a \) amounts to evaluating
\[
a^3 =
\begin{pmatrix}
g(a^2) \\
a_1 \cdot h(a^2) \\
a_2 \cdot h(a^2)
\end{pmatrix},
\]
where \( a^2 \) itself is given in terms of \( g(a) \) and \( h(a) \). Expanding this, we obtain:
\[
a^3 =
\begin{pmatrix}
g\big(g(a),\; a_1 h(a),\; a_2 h(a)\big) \\
a_1 \cdot h\big(g(a),\; a_1 h(a),\; a_2 h(a)\big) \\
a_2 \cdot h\big(g(a),\; a_1 h(a),\; a_2 h(a)\big)
\end{pmatrix}.
\]

Each component of \( a^n \) depends recursively on all components of \( a^{n-1} \), and thus, ultimately, on all components of the original vector \( a \). This leads to an intricate cross-branching functional structure, in which component-wise dependencies propagate nonlinearly across levels. The recursive process combines polynomial evaluation with functional composition, which results in increasing algebraic complexity at each iteration and makes closed-form simplification progressively more difficult.

\section{Multilinear Operation on \texorpdfstring{$K^4$}{K⁴} and Higher Dimensions}

Consider the operation \( * : V \times V \to V \), where \( V = K^4 \), defined component-wise by the following rule:
\begin{align*}
(ab)_0 &= a_0 + b_0 + a_0 b_0 + A a_1 b_1 + E a_3 b_1 + B a_2 b_2 + D a_1 b_2 + F a_3 b_2 + C a_3 b_3 , \\
(ab)_1 &= a_1 + b_1 + a_1 b_0 + a_0 b_1 + G a_1 b_1 + H a_1 b_2 + I a_1 b_3, \\
(ab)_2 &= a_2 + b_2 + a_2 b_0 + a_0 b_2 + G a_2 b_1 + H a_2 b_2 + I a_2 b_3, \\
(ab)_3 &= a_3 + b_3 + a_3 b_0 + a_0 b_3 + G a_3 b_1 + H a_3 b_2 + I a_3 b_3,
\end{align*}
where \( A, B, C, D, E, F, G, H, I \in K \) are fixed parameters.

The explicit form of the product vector \( a * a := a^2 \) is given by:
\[
a^2 =
\begin{pmatrix}
(a_0 + 1)^2 + A a^2_1 + B a^2_2 + C a^2_3 + D a_1 a_2 + E a_1 a_3 + F a_2 a_3 - 1 \\
a_1 (G a_1 + H a_2 + I a_3 + 2 (a_0 + 1)) \\
a_2 (G a_1 + H a_2 + I a_3 + 2 (a_0 + 1)) \\
a_3 (G a_1 + H a_2 + I a_3 + 2 (a_0 + 1))
\end{pmatrix}.
\]

It is worth noting that the fourth-degree construction presented here generalizes the three-dimensional version by introducing additional cross terms involving the fourth coordinate \( a_3 \). This pattern naturally extends to higher dimensions: new coordinates can be incorporated into the operation by systematically adding bilinear combinations of the new components with existing ones, following the same principles demonstrated in the \( K^3 \) and \( K^4 \) cases. In this way, one can define analogous second-order multilinear operations on \( K^n \) for arbitrary \( n \), without fundamentally altering the structure.

In particular, if we remove the fourth row and eliminate all monomials involving the index 3, we recover exactly the three-dimensional version defined earlier. Despite the conceptual simplicity of this extension, we refrain from presenting the general \( n \)-dimensional case in this work, as the resulting expressions quickly become unwieldy. Our focus remains on the 3D and 4D instances, which are sufficient to illustrate the core combinatorial and algebraic phenomena.

\section{Analysis of the Algebraic Structure}

The binary operation \( * \) defined on \( V = K^3 \) (or more generally on \( K^n \)) does not form an algebraic structure with global associativity or commutativity. This can be seen symbolically from the definition: the presence of asymmetric bilinear terms such as \( a_2 b_1 \) (without corresponding \( a_1 b_2 \)) breaks symmetry. Therefore, in general,
\[
a * b \ne b * a, \quad \text{and} \quad (a * b) * c \ne a * (b * c).
\]

The operation also lacks a neutral element for general addition-like cancellation. In particular, there is no element \( 0 \in V \) such that
\[
a * 0 = 0 * a = 0
\]
for all \( a \in V \). However, the zero vector \( e = (0, 0, 0) \in V \) plays the role of a \emph{multiplicative identity}:
\[
e * e = e, \quad a * e = e * a = a.
\]

The structure defined by \( (V, *) \) is a non-associative, non-commutative, unital magma with a distinguished identity element. Its algebraic behavior under iteration (powers) is well-defined for fixed inputs, but the general algebraic axioms (e.g., semigroup, monoid) do not hold without further restrictions.

\subsection*{Local Commutativity and Power Associativity}

Despite the lack of global commutativity and associativity in the operation \( * \), it exhibits certain well-structured behaviors when applied repeatedly to the same vector. In particular, symbolic computations show that when the operation is iterated on a fixed input \( a \in V \), the resulting powers \( a^n := \underbrace{a * a * \dots * a}_{n \text{ times}} \) behave in a commutative and power-associative manner.

\begin{definition}[Internal commutativity]
A binary operation \( * \) on a set \( V \) is said to be \emph{internally commutative} at an element \( a \in V \) if for all positive integers \( m, n \), the identity
\[
a^m * a^n = a^n * a^m
\]
holds. This property applies to powers of a single element rather than to arbitrary pairs of vectors.
\end{definition}

\begin{definition}[Power associativity]
A binary operation \( * \) on a set \( V \) is said to be \emph{power associative} if for every element \( a \in V \), the expression \( a^n := a * a * \dots * a \) (with \( n \) factors) is well-defined for all \( n \in \mathbb{N} \) (i.e., for \( n \ge 1 \)), regardless of the placement of parentheses. That is, any parenthesization of the product yields the same result.
\end{definition}

Symbolic computations indicate that our operation is power associative and internally commutative. For instance, the following expressions are symbolically equal:
\[
(a * a) * (a * a) = (a * a * a) * a = a^4; \quad \text{and} \quad ((a * a * a) * a) * a = a^5.
\]
This property ensures that exponentiation via repeated application of the operation is well-defined and unambiguous.

\begin{proposition}[Power identity]
Assume that \( * \) is power associative and internally commutative. Then for any \( a \in V \) and for all positive integers \( m, n \), we have
\[
a^m * a^n = a^{m+n}.
\]
\end{proposition}

\begin{proof}[Proof sketch]
We proceed by induction on \( m \). The base case \( m = 1 \) holds trivially by definition of \( a^{n+1} \):
\[
a * a^n = a^{n+1}.
\]
Assume that the identity holds for \( m \), i.e., \( a^m * a^n = a^{m+n} \). Then for \( m+1 \),
\[
a^{m+1} * a^n = (a^m * a) * a^n.
\]
By power associativity, we may regroup as
\[
a^m * (a * a^n) = a^m * a^{n+1}.
\]
Applying the inductive hypothesis for \( m \) and \( n+1 \) factors, we get
\[
a^m * a^{n+1} = a^{m+(n+1)} = a^{m+n+1},
\]
which completes the inductive step.
\end{proof}

\section{Hypothesis on the Absence of a Closed Form for \texorpdfstring{$a^n$}{aⁿ}}

\begin{conjecture}[No closed-form expression]
There is no general closed-form expression for the components of \( a^n := \underbrace{a * a * \dots * a}_{n \text{ times}} \) in terms of a fixed polynomial formula with finitely many terms whose structure does not depend on \( n \).
\end{conjecture}

This conjecture is supported by symbolic expansions computed for \( a^2 \) and \( a^3 \), which reveal rapid growth in both the number and degree of distinct monomials. The number of monomials in each component of \( a^n \) appears to grow combinatorially with \( n \), while the degree of the resulting polynomial increases linearly. More precisely, if \( a \in K^3 \), then each component of \( a^n \) is a multivariate polynomial in \( a_0, a_1, a_2 \) of total degree \( n \), but the number of possible monomials of degree \( n \) in three variables is
\[
\binom{n+2}{2},
\]
which grows quadratically in \( n \). Aggregated over all degrees from 1 to \( n \), the total number of monomial terms is
\[
\sum_{k=1}^n \binom{k+2}{2} = \binom{n+3}{3} - 1,
\]
indicating that the expression becomes increasingly complex with each iteration.

Furthermore, the coefficients of the monomials are not simple constants or binomial patterns: they depend intricately on the parameters \( A, B, C, D, E \) and the combinatorics of how the terms propagate through nested applications of the operation. The lack of associativity or distributive structure (in the usual algebraic sense) further complicates any attempt to compress or simplify the resulting expressions across arbitrary \( n \).

While specific cases (e.g., small \( n \), or special parameter values) may admit simplification, it appears unlikely that a uniform closed-form expression for \( a^n \) exists for arbitrary \( n \). Therefore, recursive or symbolic expansion methods remain the most viable means for studying the behavior of the sequence \( \{a^n\}_{n \in \mathbb{N}} \).

\section{Hypothesis on the Discrete Iteration Problem}

In conventional algebraic systems, such as multiplicative groups over finite fields or elliptic curves, the \emph{discrete logarithm problem} (DLP) asks: given \( g \) and \( g^n \), find \( n \). Analogously, we introduce the problem of recovering the iteration count \( n \) from the power \( a^n \) of a fixed vector \( a \in V \) under the operation \( * \). We refer to this as the \emph{Discrete Iteration Problem} (DIP).

\begin{definition}[Discrete Iteration Problem (DIP)]
Given a vector \( a \in V \) and an output \( v \in V \) such that \( v = a^n := \underbrace{a * a * \dots * a}_{n \text{ times}} \), determine the exponent \( n \in \mathbb{N} \).
\end{definition}

We hypothesize that this problem is computationally hard in general.

\begin{conjecture}[Hardness of the Discrete Iteration Problem]
Let \( * : V \times V \to V \) be the second-order multilinear operation defined above, and let \( a \in V \) be a fixed vector. Then, given \( a \) and \( a^n \), it is computationally hard to recover \( n \) for general parameter values, input vectors, and sufficiently large \( n \).
\end{conjecture}

This conjecture is strongly motivated by the apparent absence of a closed-form expression for \( a^n \), as discussed in the previous section. The number of distinct monomials in each component of \( a^n \) grows combinatorially, and the structure of the resulting polynomials becomes increasingly intricate with each iteration.

Unlike in classical algebraic groups, where exponentiation follows a single, known algebraic rule (e.g., repeated multiplication in a cyclic group), here each iteration involves recursive composition of polynomial functions whose form and coefficients change dynamically at every step. This makes direct algebraic inversion exceedingly difficult. Therefore, recovering \( n \) from \( a^n \) would typically require evaluating successive powers \( a^1, a^2, \dots, a^k \) until a match is found—an approach with exponential complexity with respect to the bit length of \( n \) in the worst case.

In this sense, the system exhibits a strong \emph{one-way} character: it is easy to compute \( a^n \) from \( a \) and \( n \), but computationally hard to reverse the process. This places it in a similar conceptual category to standard one-way functions used in cryptography, though further dedicated analysis is required to establish formal security guarantees against various attack models.

\section{Finite Fields and Cryptographic Application}

To enable practical computation and potential cryptographic deployment, we consider restricting the base field \( K \) to a finite field or ring. Two natural choices are:

\begin{itemize}
    \item The finite field \( \mathbb{F}_p = \mathbb{Z}/p\mathbb{Z} \), where \( p \) is a prime;
    \item An extension field \( \mathbb{F}_{q} = \mathbb{F}_p[x]/(f(x)) \), where \( f(x) \) is an irreducible polynomial over \( \mathbb{F}_p \).
\end{itemize}

In both cases, arithmetic in \( K^n \) remains well-defined, and all expressions involving addition and multiplication of field elements carry over without modification. Importantly, the core algebraic and combinatorial properties of the operation \( * \) — such as second-order multilinearity, power associativity, and internal commutativity — are preserved when \( K \) is replaced by a finite field. The choice of \( p \) or \( f(x) \) should be guided by security considerations (e.g., sufficiently large \( p \) to resist discrete logarithm attacks in \(\mathbb{F}_p\), or appropriate degree and irreducibility of \( f(x) \) for \(\mathbb{F}_q\)) and computational efficiency.

\subsection*{Key Exchange via Commutative Powers}

The properties of internal commutativity and power associativity allow a two-party key exchange protocol, similar in spirit to the classical Diffie–Hellman scheme, but operating over a non-associative algebraic structure.

\textbf{Protocol 1: Key Exchange}

Let \( a \in V = K^n \) be a publicly agreed base vector. Each party selects a private exponent:
\begin{itemize}
    \item Alice chooses a secret \( m \in \mathbb{N} \), computes \( A = a^m \), and sends it to Bob;
    \item Bob chooses a secret \( n \in \mathbb{N} \), computes \( B = a^n \), and sends it to Alice.
\end{itemize}

Each party then computes the shared key:
\[
K = (a^m)^n = (a^n)^m = a^{m+n},
\]
using internal commutativity and power associativity. An external observer, given \( a \), \( a^m \), and \( a^n \), would need to solve the discrete iteration recovery problem (DIP) to determine \( m \) or \( n \), which, as conjectured earlier, is computationally difficult in general. The hardness of DIP is a crucial assumption for the security of this protocol.

\subsection*{Discussion}

This construction thus leads to a natural \emph{algebraic Diffie–Hellman problem} (ADHP) over the non-associative system \( (V, *) \), defined as follows:

\begin{definition}[Algebraic Diffie–Hellman Problem]
Given a public base vector \( a \in V \) and public values \( a^m \) and \( a^n \), compute \( a^{m+n} \) without knowing either \( m \) or \( n \).
\end{definition}

The presumed hardness of ADHP stems from the lack of a closed-form expression for \( a^n \), and from the recursive, combinatorially explosive nature of the operation \( * \). Unlike the classical Diffie-Hellman protocol based on modular exponentiation, this protocol relies on the iterated composition of polynomial functions, which introduces a different type of algebraic complexity. While further cryptanalysis is required, these properties suggest a one-way behavior that may be suitable for cryptographic protocols requiring key agreement, pseudorandom generation, or iterative state evolution. Future work should investigate the resistance of this protocol to known attacks and explore potential vulnerabilities arising from the non-associative nature of the operation.

\section{Orbit Analysis of the Operator in \texorpdfstring{$V = K^3$}{V = K³}}

In this chapter, we analyze the dynamical behavior of the operator defined on the space \( V = K^3 \), where \( K = \mathbb{Z}_N \) denotes the finite field of integers modulo a prime \( N \). The binary operation \( * : V \times V \to V \) is defined componentwise as follows:
\begin{align*}
(ab)_0 &= a_0 + b_0 + a_0 b_0 + A a_1 b_1 + C a_2 b_1 + B a_2 b_2, \\
(ab)_1 &= a_1 + b_1 + a_1 b_0 + a_0 b_1 + D a_1 b_1 + E a_1 b_2, \\
(ab)_2 &= a_2 + b_2 + a_2 b_0 + a_0 b_2 + D a_2 b_1 + E a_2 b_2,
\end{align*}
where \( A, B, C, D, E \in K \) are fixed parameters. This definition supports a rich variety of algebraic structures and dynamical behaviors depending on the choice of parameters and modulus \( N \).

The primary goal of this study is to investigate the orbits of elements under repeated self-application (i.e., \( x, x*x, (x*x)*x, \dots \)). The full state space of the operator for a given modulus \( N \) consists of \( N^3 \) elements. Our focus is on the orbit lengths (cycle sizes), as well as their distribution, under various combinations of \( A, B, C, D, E \) and modulus \( N \).

\subsection{Methodology}

For each such element, we construct its orbit by repeatedly applying the operation \( x \mapsto x * x \) in the sense of left-associative exponentiation (i.e., \( x, x^2 = x * x, x^3 = (x * x) * x, \dots \)), until a previously visited state is reached, thus closing a cycle.

Special attention is paid to orbit lengths, as they reflect fundamental properties of the operator. In particular, we analyze the distribution of orbit lengths and focus on the presence and frequency of orbits of maximal length, such as \( N - 1 \) and \( N^2 - 1 \), as well as their half-lengths \( (N - 1)/2 \) and \( (N^2 - 1)/2 \).

\subsection{Key Observations and Patterns}

Experiments conducted for different values of \( N \) and parameter combinations revealed several notable patterns in the behavior of orbits.

\textbf{General behavior:}
\begin{itemize}
    \item \textit{Full coverage:} Across all tests, the operator covered the entire space \( K^3 \), with every element belonging to some orbit. This indicates well-structured dynamics across the field.
    \item \textit{Presence of long orbits:} Orbits of lengths \( N - 1 \) and \( N^2 - 1 \), as well as their halves, were observed in nearly all tested configurations. While the exact counts varied, their presence was remarkably stable across parameter sets.
\end{itemize}

\textbf{Parameter dependence:}
\begin{itemize}
    \item \textbf{Case \( N = 23 \), fixed \( C = 1, D = 1, E = 2 \):}
    \begin{itemize}
        \item An exhaustive analysis was performed over all \( A, B \in K \), yielding 462 distinct parameter combinations.
        \item \textit{Orbits of length \( N^2 - 1 \):} The proportion of such orbits ranged from 0\% to 33\%. Although rare on average (mean 2.77\%), specific combinations, such as \( (A, B, C, D, E) = (9, 19, 1, 1, 2) \), produced a notably high proportion (33\%).
        \item \textit{Orbits of length \( N - 1 \):} Much more common, these accounted for approximately 80\% of orbits on average, reaching up to 89\% (e.g., for \( (6, 1, 1, 1, 2) \)).
        \item \textit{Half-length orbits:} Orbits of length \( (N - 1)/2 \) and \( (N^2 - 1)/2 \) were observed with mean proportions of 8\% and 1\%, respectively.
    \end{itemize}

    \item \textbf{Case \( N = 61 \), representative parameters \( A = 31, B = 30, C = 1, D = 1, E = 2 \):}
    \begin{itemize}
        \item These parameters (\( A = \lfloor N/2 \rfloor + 1 \), \( B = \lfloor N/2 \rfloor \)) generated orbit distributions consistent with the averages for \( N = 23 \).
        \item The number of orbits of length \( N^2 - 1 \) was approximately \( \lfloor N/2 \rfloor \) (i.e., 31 orbits, corresponding to 0.58\% of the total), while those of length \( N - 1 \) accounted for 85\%.
        \item Initial vectors generating the longest orbits typically had the form \( (0, 1, x) \) or \( (0, 2, x) \), with small second components.
    \end{itemize}
\end{itemize}

\subsection{On the Emergence of Maximal-Length Orbits}

Of particular interest are parameter configurations that result in a significantly elevated number of \( N^2 - 1 \)-length orbits. Our experiments show that for various values of \( N \), such configurations exist and can produce up to 30\% of the total orbits at maximal length. Although no universal pattern has been identified, the phenomenon is consistent across multiple moduli.

Interestingly, even with uniform or "typical" parameter choices, the number of \( N^2 - 1 \)-length orbits often approximates \( \lfloor N/2 \rfloor \). This suggests that multiple distinct parameter sets can lead to qualitatively similar distributions, even though the actual orbits differ.

\subsection{Efficient Search for Maximal Orbits at Large \texorpdfstring{$N$}{N}}

A key practical insight is that initial vectors leading to orbits of length \( N^2 - 1 \) often have a simple structure, such as \( [0, 1, x] \) or \( [0, \text{small}, x] \). This observation enables the development of efficient probabilistic search heuristics for large \( N \), avoiding exhaustive iteration over all \( N^3 \) elements. By focusing on strategically chosen initial vectors with simple first components, one can locate long orbits with high probability and reduced computational cost.

\section{Pseudorandom Number Generation}

This section explores two distinct approaches to constructing pseudorandom number generators (PRNGs) based on the multilinear operator defined on the space \( V = K^3 \), where \( K = \mathbb{Z}_N \) is a finite field of integers modulo a prime \( N \). We contrast a traditional single-orbit approach with a novel multi-element multiplication strategy, demonstrating the superior cycle lengths and statistical uniformity achieved by the latter.

\subsection{Conventional Orbit-Based PRNG: Challenges and Observations}

A straightforward method of constructing a PRNG from the operator is to generate a sequence by iterating a single initial element under repeated self-application. That is, we compute the sequence
\[
x,\quad x * x = x^2,\quad (x * x) * x = x^3,\quad \dots
\]
using left-associative exponentiation. The orbit of \( x \) under this process serves as the pseudorandom stream, and the quality of the PRNG is determined by the orbit’s length and the statistical distribution of its components.

Empirical analysis across various moduli \( N \) and parameter sets \( (A, B, C, D, E) \) reveals several consistent challenges with this approach:

\begin{itemize}
    \item \textbf{Dominance of short cycles.} The majority of single-element orbits exhibit relatively short lengths. Approximately 80\% of orbits tend to have length \( N - 1 \), while maximal-length orbits of length \( N^2 - 1 \) are rare and typically require extensive search or parameter tuning to locate.

    \item \textbf{Non-uniform distribution.} Even when long orbits are found, the resulting sequences may exhibit statistical biases. Certain values of the components \( x_0, x_1, x_2 \) may appear more frequently than others, indicating local clustering and reducing the effectiveness of the PRNG in applications requiring uniformity.

    \item \textbf{Parameter sensitivity.} The length and structure of orbits are highly sensitive to the choice of parameters \( A, B, C, D, E \). While some configurations occasionally yield long orbits with good statistical properties, no general rule guarantees optimal behavior. This makes the single-orbit method less robust in practice.
\end{itemize}

In summary, while the single-element PRNG construction is simple and conceptually appealing, it often suffers from short cycles and suboptimal statistical behavior unless carefully tuned.

\subsection{Multi-Element PRNG: Leveraging Stratified Algebra for Enhanced Cycles}

To address the limitations of the single-orbit strategy, we propose a multi-element PRNG construction. This method exploits a key property of the operator in the context of stratified algebras: the product of elements from distinct "strata" may lead to elements in new strata, enabling a richer traversal of the state space.

Instead of computing powers of a single element via \( x \mapsto x * x \), the PRNG operates using a predefined cyclic pattern of multiplications involving multiple seed vectors. The state of the PRNG at time \( t \) is determined both by the current vector and by the current position in the multiplication pattern. This increases the effective state space from \( N^3 \) to \( N^3 \times \text{pattern\_length} \).

\begin{itemize}
    \item \textbf{Maximal cycle lengths.} Empirical evidence shows that this approach can produce cycles that approach or even achieve the theoretical maximum length. For example:
    \begin{itemize}
        \item For \( N = 37 \) and a two-element alternating pattern \texttt{[0, 1]}, the observed cycle length was 101,304, compared to the theoretical maximum of \( 37^3 \times 2 = 101,306 \).
        \item With a five-element pattern \texttt{[0, 0, 0, 1, 2]}, the resulting cycle length was 253,260, nearly reaching the maximum \( 37^3 \times 5 = 253,265 \).
    \end{itemize}
    These near-maximal pure cycles indicate that the PRNG is capable of exploring almost the entire state space without entering transient states or sub-cycles.

    \item \textbf{Uniform distribution.} Frequency histograms of the components \( x_0, x_1, x_2 \) across the generated sequences are nearly flat, with each value in \( \mathbb{Z}_N \) appearing with approximately equal probability. This suggests excellent statistical quality in the output, a critical requirement for cryptographic and simulation applications.

    \item \textbf{Parameter and seed design.} Although the multi-element method significantly improves performance, careful selection of parameters \( (A, B, C, D, E) \), seed vectors, and the multiplication pattern remains essential. However, good configurations appear more discoverable within this framework, and the likelihood of achieving long, uniform cycles increases markedly compared to the single-orbit case.
\end{itemize}

\subsection*{Summary}

The multi-element PRNG framework represents a substantial advancement over the conventional single-orbit strategy. By introducing structured variation through multiple seed elements and cyclic multiplication patterns, it enables generation of long-period, uniformly distributed sequences with high reliability. These properties are direct consequences of the stratified structure and nonlinear combinatorics of the underlying operator. Future work will focus on theoretical models explaining these empirical results, formal security analysis, and broader exploration of parameter configurations across different finite fields.

\section*{Code Availability}
The Python implementation of the multilinear operations, including the M3 and M4 examples, and the key exchange protocol, is available under an MIT License at the following GitHub repository: \url{https://github.com/stas-semenov/one-way-multilinear/}.

\section*{Conclusion}

In this work, we introduced and analyzed a novel class of binary operations on finite-dimensional vector spaces over a field \( K \), defined by second-order multilinear forms with linear shifts. These operations exhibit rapidly increasing algebraic complexity under iteration: while the polynomial degree of \( a^n \) grows linearly, the number of distinct monomials expands combinatorially.

Despite being non-associative and non-commutative in general, the operations satisfy power associativity and internal commutativity when applied repeatedly to a single vector. These structural properties guarantee that the iterated product \( a^n \) is well-defined. A central conjecture is the absence of a closed-form expression for \( a^n \), implying a potential one-way behavior: computing \( a^n \) is efficient, while recovering \( n \) from \( a^n \) — the Discrete Iteration Problem (DIP) — appears computationally hard.

To further explore the implications of this complexity, we conducted a detailed empirical study of orbit structures in the space \( V = K^3 \) over finite fields \( \mathbb{Z}_N \). The analysis revealed that the operator induces a full partition of the space into disjoint orbits, with orbit lengths strongly dependent on the choice of parameters \( A, B, C, D, E \), but largely independent of the modulus \( N \). Orbits of lengths \( N - 1 \) and \( N^2 - 1 \) consistently appear, often with surprisingly stable frequencies. Notably, the number of maximal-length orbits (length \( N^2 - 1 \)) is often close to \( \lfloor N/2 \rfloor \), supporting the conjecture that the orbit structure reflects deep algebraic complexity even for moderate field sizes.

These observations reinforce the view that the proposed operation family exhibits inherent one-wayness, making it a viable foundation for cryptographic applications. We presented a Diffie–Hellman-like key exchange protocol based on this behavior, leading to the formulation of the Algebraic Diffie–Hellman Problem (ADHP).

In addition, we proposed a pseudorandom number generation (PRNG) method based on multi-element multiplication patterns, which empirically achieves near-maximal cycle lengths and excellent uniformity across state components. This approach addresses key limitations of traditional orbit-based PRNGs and demonstrates the practical utility of stratified algebra in state space traversal.

Future work includes formal cryptanalysis of ADHP, exploration of new primitives such as digital signatures or pseudorandom generators, and deeper theoretical investigation into the algebraic dynamics of the system — particularly regarding orbit structures, parameter sensitivity, and their interplay with computational hardness.


\begin{thebibliography}{1}

\bibitem{Diffie1976}
Whitfield Diffie and Martin~E. Hellman.
\newblock New directions in cryptography.
\newblock {\em IEEE Transactions on Information Theory}, 22(6):644--654, 1976.

\bibitem{Goldreich2001}
Oded Goldreich.
\newblock {\em Foundations of Cryptography: Volume 1, Basic Tools}.
\newblock Cambridge University Press, Cambridge, UK, 2001.

\bibitem{semenov2025stratified}
Stanislav Semenov.
\newblock Stratified algebra.
\newblock {\em arXiv preprint arXiv:2505.18863}, 2025.

\end{thebibliography}

\end{document}